\RequirePackage{fix-cm}
\documentclass[smallextended]{svjour3}       
\smartqed  
\usepackage{
amsmath,
amssymb,
hyperref,
mathrsfs,
graphicx,
stmaryrd,
amscd,
 amsfonts, stmaryrd,latexsym, xargs}
\usepackage{color}

\usepackage[colorinlistoftodos]{todonotes}
 \presetkeys{todonotes}%
{inline,backgroundcolor=gray!20,bordercolor=gray!30}{}
\tikzset{/tikz/notestyleraw/.append style={text=black}}

\newcommand\todoin[2][]{\todo[inline, caption={2do}, #1]{
 \begin{minipage}{\textwidth-3pt}#2\end{minipage}
 }}

\newtheorem{thm}{Theorem}[section]

\newtheorem{lem}[thm]{Lemma}
\newtheorem{defn}[thm]{Definition}
\newtheorem{que}{Question}

\newtheorem{prop}[thm]{Proposition}
\newtheorem{cor}[thm]{Corollary}
\newtheorem{rmk}[thm]{Remark}
\newcommand{\be}{\begin{eqnarray}}
\newcommand{\ee}{\end{eqnarray}}
\newcommand{\ben}{\begin{eqnarray*}}
\newcommand{\een}{\end{eqnarray*}}

\newcommand{\beal}{\begin{aligned}}
\newcommand{\enal}{\end{aligned}}
\newcommand{\bc}{\begin{cases}}
\newcommand{\ec}{\end{cases}}

\newcommand{\Z}{\mathbb{Z}}

\newcommand{\R}{\mathbb{R}}

\newcommand{\Q}{\mathbb{Q}}
\newcommand{\A}{\mathbb{A}}

\newcommand{\N}{\mathbb{N}}

\newcommand{\cB}{\mathcal{B}}
\newcommand{\cC}{\mathcal{C}}

\newcommand{\cS}{\mathcal{S}}
\newcommand{\cR}{\mathcal{R}}

\newcommand{\cI}{\mathcal{I}}

\newcommand{\cT}{\mathcal{T}}
\newcommand{\cA}{\mathcal{A}}

\newcommand{\cM}{\mathcal{M}}
\newcommand{\cL}{\mathcal{L}}

\newcommand{\Om}{\Omega}
\newcommand{\eps}{\epsilon}

\newcommand{\Ol}{\overline}

\newcommand{\om}{\omega}

\newcommand{\dt}{\delta}

\newcommand{\T}{\mathbb{T}}

\newcommand{\baru}{\bar{u}}
\newcommand{\barv}{\bar{v}}

\begin{document}

\title{Spectral invariants of convex billiard maps: 
}
\subtitle{a viewpoint of Mather's $\beta-$function }


\author{Jianlu Zhang         
}


\institute{
              Hua Loo-Keng Key Laboratory of Mathematics \& Mathematics Institute, Academy of Mathematics and systems science, Chinese Academy of Sciences, Beijing 100190, China \\
              Tel.: +86-182-1038-3625\\
              \email{jellychung1987@gmail.com}           
}

\date{Received: date / Accepted: date}

\maketitle

\begin{abstract}
For strictly convex billiard maps of smooth boundaries, we get a Birkhoff normal form via a list of constructive generating functions. Based on this, we get an explicit formula for the $\beta-$function (locally), and explored the relation between the spectral invariants of the billiard maps and the $\beta-$function.
\end{abstract}

\section{Introduction}\label{s1}

The billiard map describes the frictionless motion of a massless particle inside a smoothly bounded region $\Om\subset\R^2$ with elastic reflections at the boundary $\partial\Om$. When the particle hits the boundary, it  conforms to the law of optical reflection: the angle of reflection equals the angle of incidence, see Fig. \ref{fig1}. When the boundary $\partial\Om$ is strictly convex, the motion of the particle can be totally interpreted as trajectories of a {\sf symplectic twist maps}, see \cite{M}. So we can convert the billiard geometric problems to dynamic problems of twist maps. Especially, as the introducing of variational approaches by Mather in 1980s, the prospect of this topic get totally refreshed and a lot of inspiring conclusions were revealed \cite{ADK,HKS,KS,Ma,Z,Z2}.\\

Although the dynamic essence of convex billiard maps is comparably straightforward, its qualitative properties are extremely nonlocal. Precisely, the geometric state of $\partial\Om$ will decide several intriguing rigidity phenomena, which form the basis of the famous {\sf Marked Length Spectrum Conjecture} and {\sf Birkhoff Conjecture}. We will devote the whole paper to show how these conjectures relates with the Mather's $\beta-$function.

\subsection{Dynamics of convex billiard maps}\label{s1.1}

Suppose that $\partial\Om=\xi(s)$ is parametrized by {\sf arc length} $s$ and the perimeter equals $1$. Let $v$ be the {\sf incident angle} starting from $s$, then the billiard map can be identified by a diffeomorphism on the annulus $\mathbb A = [0, 1] \times [0,\pi]$ (see Fig. \ref{fig1}):
\be\label{eq:b-m}
\phi:(s,v)\rightarrow(\Ol s,\Ol v),
\ee
which can be formalized by the following {\sf generating function} in the universal covering space
\be\label{eq:gene-0}
h(x,\overline x):=-\big\|\xi([x])-\xi([\Ol x])\big\|, \quad\quad \forall(x,\Ol x)\in\R^2
\ee
with $x\equiv [x] (\text{mod }1)$ and $\Ol x\equiv [\Ol x] (\text{mod }1)$.
Indeed, we have 
\be\label{p-h}
\partial_1 h=\cos v,\quad \partial_2h=-\cos \Ol v.
\ee
Notice that $-\partial_{12}h>0$ once the boundary is strictly convex \cite{M2}, and the diffeomorphism $\phi$ preserves the area form $\om=\sin vdv\wedge dx$, which is {\sf exact} since $d\alpha=\om$ for $\alpha=-\cos vdx$.
\begin{rmk}
Here the exactness of $\om$ is necessary, to guarantee the {\sf marked length spectrum} is a symplectic invariant of the map $\phi$ (see Sec. \ref{s1.2} for definition). Moreover, the exactness should be kept during the iterations of Birkhoff transformations, such that the marked length spectrum still applies for the final Birkhoff normal form.
\end{rmk}

\begin{figure}
\begin{center}
\includegraphics[width=6cm]{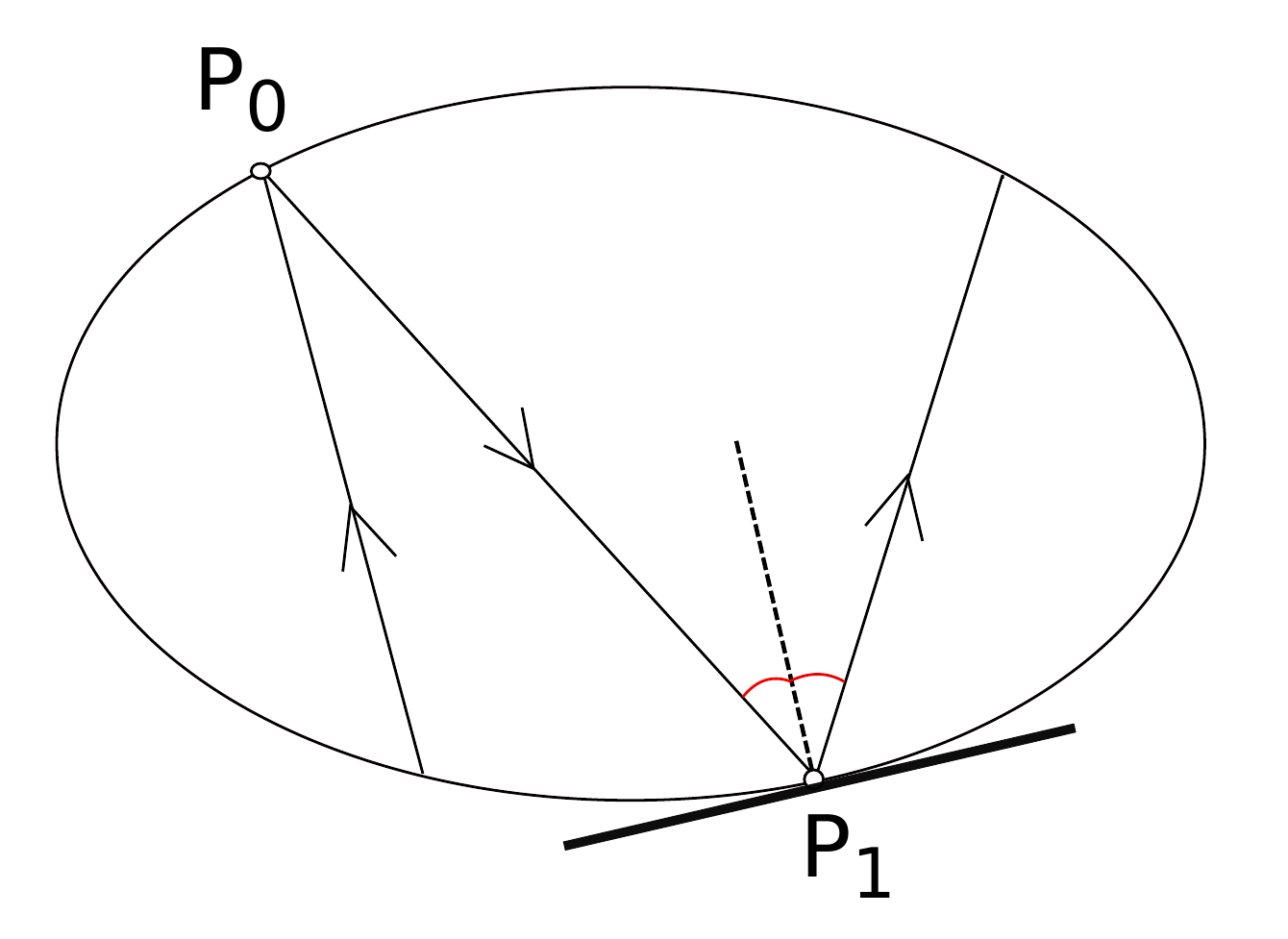}
\caption{The reflective angle keeps equal to the incident angle for every rebound.}
\label{fig1}
\end{center}
\end{figure}

\begin{lem}[{\sf Reversibility}]
The billiard map (\ref{eq:b-m}) is reversible, in the sense that $\phi\circ \cR\circ\phi=\cR$ for the following defined {\sf reflective diffeomorphism}:
\be\label{eq:sym-1}
\cR:(s,v)\in\A\rightarrow (s,\pi-v)\in\A.
\ee
\end{lem}
\begin{rmk}
The reversibility forces the $\beta-$function of the billiard map to be of a special form. During the interation of Birkhoff transformations it is also preserved.
\end{rmk}

\subsection{Periodic orbits and length spectrum}\label{s1.2}

Finding periodic orbits and explore their properties is naturally the foremost dynamic
feature we concern.As an example, Birkhoff's application on {\sf Poincare's Last Geometric Theorem} proves the existence of at least two period$-n$ orbits for each $n\geq 2$ \cite{B}. However, the amount of periodic orbits in $\A$ is much bigger than that. Here we propose a variational approach to explain why. Let's introduce a notion {\sf rotation number} first: 
\begin{defn}
Assume the positive orientation of $\partial\Om$ be the direct $s$ raises, then the {\sf winding number} of a periodic billiard orbit $\gamma$ is defined by the number of how many times $\gamma$ goes around $\partial\Om$ in the positive direction until it returns to the same initial position. The {\sf rotation number} is defined and denoted by 
\be\label{eq:rot-num}
\rho(\gamma):=\frac{\text{winding number of $\gamma$}}{\text{period of $\gamma$}}\in(0,\frac12].
\ee
The reason we restrict the rotation number within $(0,1/2]$ is because $\cR(\gamma)$ has a rotation number $1-\rho(\gamma)$ due to the reversibility and vice versa.
\end{defn}

Now for any $p/q\in\Q$ in lowest terms, we can define the following {\sf barrier function}
\be
\cB(s,p/q):=\min_{X\in\cC(s,p/q)}\sum_{i=0}^{q-1} h(x_i,x_{i+1}),\quad s\in[0,1)
\ee
where $\cC(s,p/q)$ is the space of all $p/q-$configurations $\{x_i\}_{i=0}^{q-1}$ with $x_0\equiv s$. The minimizer and maximizer of $\cB(s,p/q)$ will individually decides a $p/q-$periodic orbit (Thm 1.2.4 of \cite{Si}).
\begin{defn}
For any convex billiard domain $\Om$, the {\sf marked length spectrum} is defined by 
\be
\cM\cL(\Om):=\{\min_{s\in[0,1]}\cB(s,p/q)|p/q\in\Q\cap(0,1/2]\}.
\ee
\end{defn}
\begin{rmk}\label{rmk:sym-inv}
The significance of the marked length spectrum is the {\sf symplectic invariance}. Recall that 
\[
\phi^*\alpha-\alpha=d h(x,\Ol x),
\]
for any exact symplectic transformation $\Phi:\A\rightarrow\A$, the following commutative diagram
\[
\begin{CD}
(x,v)@>\phi>>(\Ol x,\Ol v)\\
@AA\Phi A @ AA\Phi A\\
(\theta,I)@>\psi>>(\Ol\theta,\Ol I),
\end{CD} 
\]
 is achievable and the billiard map $\phi$ will be transformed into $\psi$. If $\Phi$ has an associated generating function $g(x,\theta)$, i.e.
\[
dg=\beta-\alpha
\]
where $\beta$ is a new $1-$form on $\A$ with $d\beta=\om$, then $\psi$ has to be generated by the following 
\be
\hbar(\theta,\Ol\theta)=h(x,\Ol x)+g(x,\theta)-g(\Ol x,\Ol \theta).
\ee
Therefore, for any periodic orbit $X=\{x_i\}_{i=0}^{q-1}$, $\Phi$ will send it into a new periodic orbit $\Theta=\{\theta_i\}_{i=0}^{q-1}$ but the {\sf length function} 
\be
A(\Theta)&:=&\sum_{i=0}^{q-1}\hbar(\theta_i,\theta_{i+1})\nonumber\\
&=&\sum_{i=0}^{q-1}h(x_i,x_{i+1})+\sum_{i=0}^{q-1}g(x_i,\theta_{i})-\sum_{i=0}^{q-1}g(x_{i+1},\theta_{i+1})\nonumber\\
&=&\sum_{i=0}^{q-1}h(x_i,x_{i+1}):=A(X)
\ee
keeps the same.
\end{rmk}

Previous analysis shows that for any two different billiard domains $\Om$ and $\Om'$, if the dynamics of them are exact symplectic diffeomorphic to each other, they have the same marked length spectrum; Conversely, it's natural to raise the following question
\begin{que}[{\sf Marked Length Spectrum Conjecture \cite{GM}}]\label{que:mls-conj}
For any two different convex billiard domains $\Om$ and $\Om'$ which are isospectral in the sense of $\cM\cL(\Om)=\cM\cL(\Om')$, is it true that $\partial\Om$ is isometric to $\partial\Om'$ ?
\end{que}

\begin{rmk}
In \cite{GM} the authors also discussed the relations between the length spectrum and the eigenvalues of such a Dirichlet problem:
\be\label{eq:dirichlet}
 \left\{
 \begin{array}{cccccccccc}
 \Delta u&=&\lambda u&\quad\text{in\;} \Om\subset\R^2,\\
 u|_{\partial\Om}&=&0.& \quad\quad
 \end{array}
 \right.
\ee
They showed that we can recover the marked length spectrum from the information of those eigenvalues, and the meaning of latter can be traced to a famous problem: {\tt Can one hear the shape of a drum?} This problem is formulated by Mark Kac in \cite{K} (actually had been earlier stated by Hermann Weyl), which aims to establish the boundary $\partial\Om$ purely from the eigenvalues (so called {\sf Laplace spectrum}). Although we have gotten both negative answers in \cite{Mi} and positive ones in \cite{HZ,S}, this problem is still far from being completely solved.
\end{rmk}
 
\subsection{Integrability of billiards}\label{s1.3}

 We have different versions of `integrability' for the convex billiards, depending on the occasion we mention it. Dynamically, the {\sf Liouville Integrability} imposes the whole phase space of the map $\phi$ is foliated by non-contractive invariant curves ({\sf caustics}), and Bialy has proved circular billiard is the only Liouville integrable one \cite{Bi}. If we loose the global foliation to local, e.g. a positive measure set, we get a weaker integrability and obviously elliptic billiard is one. 
 \begin{que}[{\sf Birkhoff Conjecture}]\label{que:b-conj}
 There is no other local integrable billiard except for the elliptic one.
 \end{que}
 Recently, the deformative version of this Conjecture gets a great breakthrough, in a series of works by Kaloshin and his collaborators \cite{ADK,HKS,KS}. Moreover, we have chance to shrink the `local condition' to just 2 periodic caustics, $1/2$ and $1/3$ ones. Partial results has been achieved in \cite{Z2}. Nonetheless, the thorough resolution of this Conjecture is still far to reach. We will explain how the Mather theory support us on this Conjecture and exhibit some novel aspects of this topic in this paper.

\subsection{Mather's theory on convex biliards} \label{s1.4}

The original purpose of the Mather theory, is to explore the minimizing configurations of twist maps in an asymptotic viewpoint, then classify them by different homology (or cohomology) classes \cite{Ma2}. Based on this, we can find a list of invariant sets with different rotation numbers, and construct the heteroclinic (resp. homoclinic) orbits among them. Since in the current paper we concern only global geometric problems of billiard maps, of which individual trajectory conformation is less interested. So we will only display some necessary notions in this subsection.\\

A bi-infinite configuration $X=\{x_i\in\R\}_{-\infty}^{+\infty}$ is called {\sf minimal configuration}, if any finite segment of $X$ is minimal by fixing the two ending points. For any minimal configuration $X$, we can verify the {\sf Euler-Lagrange equation}:
\be
\partial_1 h(x_i,x_{i+1})+\partial_2 h(x_{i-1},x_i)=0,\quad\forall i\in\Z,
\ee
which implies $\{([x_i],-\partial_1h(x_i,x_{i+1}))\}_{-\infty}^{+\infty}$ is an orbit of the billiard map $\phi$. Moreover, there exists a uniquely identified rotation number for it, which can be formalized by 
\be
\rho(X):=\lim_{i\rightarrow\pm\infty}\frac{x_i}i.
\ee
Notice that this definition matches (\ref{eq:rot-num}) for periodic minimal configurations.

\begin{defn}[$\beta-$function]
For any $p/q\in\Q\cap(0,1/2]$ in lowest terms, the associated $\beta-$value is defined by 
\[
\beta(p/q):=\frac1 q\min_{s\in[0,1]}\cB(s,p/q).
\]
For any irrational number $\om\in(0,1/2]$, it's $\beta-$value can be defined in a limit way, i.e.
\[
\beta(\om):=\lim_{n\rightarrow+\infty}\beta(p_n/q_n)
\]
for any rational sequence $\{p_n/q_n\}_{n\in\N}$ approaching $\om$.
\end{defn}
\begin{rmk}
Any ${np}/{nq}-$minimizing configuration $X$ has to be $p/q-$minimizing and vice versa. This is because the minimizing periodic orbit has no self-intersection \cite{M2}. This guarantees $\beta(p/q)$ is well defined; Besides, as $p_n/q_n\rightarrow \om$, the associated minimal configuration $X_n$ will converge to a $\om-$minimal configuration $X_\om$, which accordingly ensures the well-definiteness of $\beta(\om)$.
\end{rmk}
\begin{prop}\cite{Ma3} \label{prop:beta}
Here we display several useful properties of $\beta:(0,1/2]\rightarrow\R$:
\begin{enumerate}
\item $\beta(h)$ is strictly convex for $h\in (0,1/2]$;
\item $\beta(h)$ is differentiable at $h\in\Q\cap(0,1/2]$, if and only if there exists a $h-$caustic for the billiard map. Accordingly, for generic billiard maps, $\beta(h)$ is not differentiable at all $h\in\Q\cap (0,1/2]$;
\end{enumerate}
\end{prop}
\begin{thm}[{\sf Preliminary}]
The $\beta-$function of convex billiard maps is a symplectic invariance.
\end{thm}
\begin{proof}
For any $h\in\Q\cap(0,1/2]$, $\beta(h)$ is invariant under exact symplectic transformations due to Remark \ref{rmk:sym-inv}. For irrational $0<h<1/2$, $\beta(h)$ is the limit of $\beta(p_n/q_n)$ with $\{p_n/q_n\}_{n\in\N}$ approaching to $h$, so $\beta(h)$ is also symplectic invariant. 
\end{proof}

Notice that the similar symplectic invariance of general convex Hamiltonian systems has been proved by Bernard in \cite{B}. However, it's usually unknown wether the Taylor expansion can be achieved for $\beta$ (see Theorem \ref{thm:m-t}). On the other side, other equivalent symplectic invariance of the convex billiards have been given in \cite{GM,MM} by geometric methods. 

As a new symplectic invariance, now we substitute the role of marked length spectrum by the $\beta-$function, and reform previous Conjectures  by the following:
\begin{que}\cite{So}
For any two different convex billiard domains $\Om$ and $\Om'$ satisfying $\beta_{\Om}=\beta_{\Om'}$, is it true that $\partial\Om$ is isometric to $\partial\Om'$ ?
\end{que}

\begin{que}\cite{So}
For billiard domain $\Om$ with smooth $\beta_\Om$ in $(0,\eps)$ for some $\eps>0$, is $\Om$ an ellipse?
\end{que}

These two questions are equivalent translation of previous Marked Length Spectrum Conjecture and Birkhoff Conjecture respectively. Indeed, due to (2) of Proposition \ref{prop:beta}, the local integrability of a billiard map implies $C^1$ smoothness of $\beta-$function in certain interval. Actually, evidence from \cite{MM} indicates the possibility to identify the billiard boundary $\partial\Om$ via the Taylor expansion of $\beta-$function near $0$, from a viewpoint of PDE. Basically, they constructed a {\sf wave trace} function $\tau(t)$ for (\ref{eq:dirichlet}), of which the {\sf singular support set} is contained in the length spectrum set. For smooth boundary, the wave trace has an asymptotic expansion as $t\rightarrow 0^+$, and the coefficients works as symplectic invariants. The first few coefficients are known to have geometric meanings, but the rest coefficients are hard to be determined and lost the geometric visualization.\\

We will use the following heuristic deduction to explain the importance of finding a substitutive symplectic invariance, say the $\beta-$function. For a $C^\infty$ (or $C^\om$) smooth billiard boundary $\partial\Om$, which can be parametrized by $\gamma(s)$ with an arc length  $s\in[0,1]$, it will be uniquely identified by the {\sf curvature radius} $\rho(s)$ with
\be
\rho(s)=\frac1 {|\ddot\gamma(s)|},\quad s\in[0,1].
\ee
Apparently the Fourier expansion of $\rho(s)=\sum_{k\in\Z}\rho_k e^{i2\pi ns}$ will uniquely identify the boundary $\gamma(s)$. On the other side, the $\beta-$function holds a Taylor expansion  
\[
\beta(h)=\sum_{n\in\N}\beta_n h^n
\]
where $\{\beta_n\}_{n\in\N}$ present as symplectic invariants. Intuitively once the correspondence between these two 
\[
\{\rho_k\}_{k\in\Z}\slash\cI\in\R^\infty\xrightarrow{\mathfrak F} \{\beta_n\}_{n\in\N}\in\R^\infty
\]
is identified to be a homeomorphism, then we can totally recover the boundary $\partial\Om$ via the $\beta-$function (here $\cI$ is the isometric transformation).

\subsection{Main result.}\label{s1.5}
 Inspired by the normal form iterations of \cite{LM}, we gift a Birkhoff normal form for the billiard map $\phi$, in a region close to the boundaries of $\A$. Benefit from this normal form, all the coefficients of the Taylor expansion of $\beta(h)$ can be established  and the relation with $\rho(s)$ can be explicitly observed.

\begin{thm}[{\sf Main 1}]\label{thm:m-t}
For strictly convex billiard domain $\Om$ with $C^\infty$ (or $C^\om$) smooth boundary $\partial\Om$, there exists a $C^\infty$ exact symplectic diffeomorphism $\Phi:\A\rightarrow\A$ such that the billiard map $\phi$ in (\ref{eq:b-m}) can be transformed into a Birkhoff normal form as
\be\label{eq:birk-norm-form}
\psi=:\Phi^{-1}\circ\phi\circ\Phi:\A\rightarrow\A,\quad\text{via }\left\{
\begin{aligned}
&x^+=x+\zeta_\infty(y),\\
&y^+=y_\infty,
\end{aligned}
\right.
\ee
of which formally
\be\label{eq:tar-thm}
 \zeta_\infty(y)=y+\sum_{i=1}^{\infty}c_{2i+1}y^{2i+1}
\ee
as $y\rightarrow 0$ being an infinitesimal.
\end{thm}
\begin{rmk}
\begin{enumerate}
\item The proof of this Theorem is iterative, of which each step transformation can be achieved from an explicit generating function, so $\psi$ inherits the same $\beta(h)$ as $\phi$. 
\item This Theorem tells us the arbitrarily high order jet of $\psi$ at the border $\{y=0\}$. Although (\ref{eq:birk-norm-form}) is not a normal form in any belt region $\{(x,y)\in\T\times[0,\dt)\}$ with $\dt\ll1$ (even for $C^\om-$smooth $\partial\Om$), it indeed supplies us with an effective expansion of the $\beta-$function.
\end{enumerate}
\end{rmk}
\begin{cor}[{\sf Main 2}]\label{cor:1}
The $\beta-$function of the billiard map can be formalized by 
\be\label{eq:beta}
\beta(h)=\frac{\sqrt2}3c^{3/2}+\sum_{i=1}^{+\infty}\frac{2i+1}{2i+3}c_{2i+1}\sqrt{2^{2i+1}}c^{\frac{2i+3}{2}}
\ee
with
\[
c=\frac12[\zeta_\infty^{-1}(h)]^2
\]
is also symplectic invariant. Here $\zeta_\infty^{-1}(h)$ is the inverse function of $h=\zeta_\infty(c)$.
\end{cor}
\begin{rmk}
Notice that for $c\ll1$, we can solve (\ref{eq:beta}) by a {\sf method of undetermined coefficients}, to formally express $\beta(h)$ by
\be\label{beta-formal}
\beta(h)=\sum_{n=0}^{+\infty}\beta_{2n+1} h^{2n+1}.
\ee
Then the following easy consequence can be deduced:
\end{rmk}
\begin{cor}[{\sf Main 3}]\label{cor:2}
Let $\Om$ be a strictly convex domain with smooth boundary, then
\[
\beta_3+\pi^2\beta_1\leq 0
\]
and equality holds if and only if $\Om$ is a disc.
\end{cor}

In Sec. 8 of \cite{MM}, this inequality is essentially proved by the isoperimetric inequality. Also in \cite{Si,So} they gave similar results in their individually settings. Actually, due to (\ref{sr}), (\ref{generating}) and Lemma \ref{lem:time-rev}, this inequality can be shown in the Lazutkin's coordinate, which is given as our first step iteration in (\ref{eq:diag}).

\subsection{Outlook: a comparison with related works}

The novelty of the current paper is that we proposed a constructive way to obtain the Birkhoff normal form for the convex billiards which is symplectic, although Theorem \ref{thm:m-t} has been revealed in \cite{MM} via a spectral approach. The readers can find in Lemma \ref{lem:time-rev} an inductive formula ((\ref{model}) and (\ref{billiard1})), and detailed constructive equations has been given by the gray-shadowing expressions. Theoretically, starting from the Lazutkin's coordinate, this iterative scheme helps us to solve the coefficients $c_{2i+1}$ in (\ref{eq:tar-thm}) for arbitrarily large $i$, which accordingly solves all the coefficients of the $\beta-$function. In \cite{So}, the author ever solved the $\beta_n$ till order $9$, by using a computational assistance. \\

On the other side, we want to point out, a similar KAM iteration has been used in \cite{P} to show the existence of caustics with Diophantine rotation numbers, by using the approximated interpolating Hamiltonians. In \cite{Z}, the author successfully embedded the billiard map into a convex time-periodic Hamiltonian flow, and got the KAM theorem as a corollary of the classical KAM theory for Hamiltonian systems. However, the difference between those KAM iterations and the current paper is that the transformations are made w.r.t. different frequencies. In this paper the iterations are all taken in a region suitably close to the billiard boundary, so the effective expansion of $\beta$ is of $h=0$. Nonetheless, the KAM theorem in \cite{P} can be expected to get other expansions of $\beta$, but of Diophantine $h\neq 0$.

\begin{que}
For any two different convex billiard domains $\Om$ and $\Om'$ satisfying $\beta_{\Om}(h)=\beta_{\Om'}(h)$ for $h\in(0,\eps)\cap\mathcal D_{\alpha}^{\tau}$ for some $\eps>0$ and
\[
\mathcal D_{\alpha}^\tau:=\big\{\om\in\R|\forall (k,l)\in\Z^2\backslash\{0\},\ |k\om+l|\geq \frac{\alpha}{|k|^\tau}\big\}
\]
 is it true that $\partial\Om$ is isometric to $\partial\Om'$ ? 
\end{que}

\subsection{Organization of the article.}\label{s1.6}

 In Section \ref{s2}, we gave a detailed iterative scheme of the Birkhoff normal forms. Benefit from this, we proved Theorem \ref{thm:m-t}. In Section \ref{s3}, we concluded the Taylor expansion of $\beta(h)$ and proved associated Corollary \ref{cor:1} and Corollary\ref{cor:2}. 

\vspace{10pt}

\noindent{\bf Acknowledgement.} The work is supported by the National Natural Science Foundation of China (Grant No. 11901560). The Author is indebted to Prof. Vadim Kaloshin and Prof. Ke Zhang for helpful discussions, and is also grateful to Prof. de Simoi and Prof. Sorrentino for their  support of relevant computational materials.

\vspace{20pt}

\section{Iterative construction of the Birkhoff normal form}\label{s2}

\vspace{20pt}

As a warmup, let's review the Lazutkin's coordinate and take it as the initial step of iteration \cite{L}. For sufficiently small reflected angle $0<v\ll1$, the billiard map can be expressed by
\begin{eqnarray}\label{eq1}
\phi:\left\{
\begin{aligned}
&s'=&s+\alpha_1(s)v+\alpha_2(s)v^2+\alpha_3(s)v^3+F(s,v)v^4,\\
&v'=&v+\beta_2(s)v^2+\beta_3(s)v^3+G(s,v)v^4,
\end{aligned}
\right.
\end{eqnarray}
where
\be
\alpha_1(s)=2\rho(s),\;\alpha_2(s)=\frac{4}{3}\rho(s)\dot\rho(s),\nonumber\\
\alpha_3(s)=\frac{2}{3}\rho^2\ddot\rho+\frac{4}{9}\rho\dot\rho^2,\nonumber\\
\beta_2(s)=-\frac{2}{3}\dot\rho,\;\beta_3(s)=-\frac{2}{3}\rho\ddot\rho+\frac{4}{9}\dot\rho^2,\nonumber
\ee
with $\rho(s)$ being the curvature radius of $\partial\Omega$. By  applying the Lazutkin's transformation 
\be
\Phi: x =\frac{\int_0^s\rho(\tau)^{-2/3}d\tau}{\int_0^1\rho(s)^{-2/3}ds},\quad y =\dfrac{4\rho(s)^{1/3}\sin v/2}{\int_0^1\rho(s)^{-2/3}ds}
\ee
 the map (\ref{eq1}) becomes
\begin{eqnarray}\label{sr}
\varphi:\left\{
\begin{aligned}
\Ol x&=&x+y+y^3f(x,y),\\
\Ol y&=&y+y^4g(x,y),
\end{aligned}
\right.
\end{eqnarray}
which preserves the symplectic form $d({y^2}/{2}) \wedge dx$. The generating function has the form
\begin{equation}\label{generating}
\hbar(x,\Ol x)=4C_1^2\int_x^{\Ol x}\rho^{2/3}(s(\tau))d\tau+4C_1^3h(s,s'),
\end{equation}
with $h(s,s')$ given in (\ref{eq:gene-0}) and
\[
 C_1=\Big{(}\int_0^1\rho^{-2/3}(s)ds\Big{)}^{-1}.
\]
In other words, we have
\[
d\hbar(x,\Ol x)=\frac{y'^{2}}2d\Ol x-\frac {y^2}2dx
\]
and $\om=d\frac {y^2}2\wedge dx=d\alpha$ with $\alpha=\frac {y^2}2dx$. That implies $\varphi$ is a canonical symplectic diffeomorphism with the standard symplectic $2-$form. That's the foundation of applying Birkhoff iterations in the following paragraph.
\begin{rmk}
Notice that (\ref{sr}) is nearly integrable map for $0<l\ll1$, benefit from this property Lazutkin proved the existence of a cantor set of caustics by the KAM theorem \cite{L}.
\end{rmk}
Another crucial observation is that we need to inherit the `reversibility' in the Lazutkin coordinate:
\begin{lem}[{\sf Mirror Symmetry}]
The new billiard map $\varphi$ is reversible in the sense that $\varphi\circ\cT\circ\varphi=\cT$ for the mirror reflection 
\be\label{eq:sym-2}
\cT:(x,y)\in\cA\rightarrow (x,-y)\in\cA.
\ee
\end{lem}
\begin{proof}
Notice that $\varphi=\Phi\circ\phi\circ\Phi^{-1}$, $\Phi\circ\cT=\cT\circ\Phi$, $\cR\circ\cR=Id$ and $\phi\circ\cR\circ\phi=\cR$ due to (\ref{eq:sym-1}). That implies
\be
\varphi\circ\cT&=&\Phi\circ\phi\circ\Phi^{-1}\circ\cT=\Phi\circ\phi\circ\cT\circ\Phi^{-1}\nonumber\\
&=&\Phi\circ\phi\circ\cS^{-1}\circ\cR\circ\Phi^{-1}\nonumber\\
&=&\Phi\circ\cT\circ\cR^{-1}\circ\phi\circ\cR\circ\Phi^{-1}\nonumber\\
&=&\cT\circ\Phi\circ\cR^{-1}\circ\phi\circ\cR\circ\Phi^{-1}\nonumber\\
&=&\cT\circ\Phi\circ\phi^{-1}\circ\cR^{-1}\circ\cR\circ\Phi^{-1}\nonumber\\
&=&\cT\circ\Phi\circ\phi^{-1}\circ\Phi^{-1}=\cT\circ\varphi^{-1}
\ee
where the operator $\cS(s,v)=(s,v+\pi)$ satisfying $\cS\circ\cT=\cR$.
\end{proof}
\begin{rmk}
Later we will see the mirror symmetry will be preserved in the Birkhoff iterations.
\end{rmk}

\subsection{Coordinate Transition Chain} We can formalize the iteration by the following commutative diagram: Suppose we have a transition chain like
\be\label{eq:diag}
\begin{CD}
(x,y)@>\Phi_0>>(x_1,y_1)@>\Phi_1>>\cdots @>\Phi_{n-1}>>(x_n,y_n)@>\Phi_n>>\\
@V\varphi VV @ V\varphi_1 VV @ V  VV @ V\varphi_n VV  \\
(\Ol x,\Ol y)@> \Phi_0>>(\Ol x_1,\Ol y_1)@>\Phi_1>>\cdots @>\Phi_{n-1}>>(\Ol x_n,\Ol y_n)@>\Phi_n>>
\end{CD}
\ee
which transforms (\ref{sr}) into a limit map:
\[
\varphi_\infty:\left\{
\begin{aligned}
&x_\infty^+=x_\infty+\zeta_\infty(y_\infty),\\
& y_\infty^+=y_\infty,
\end{aligned}
\right.
\]
defined for $(x_\infty,y_\infty)\in \T\times [0,r)$ with $r\ll1$. Besides, during each step we have
\[
y_{k+1} dy_{k+1}\wedge dx_{k+1}=y_kdy_k\wedge dx_k,\quad\forall k\in\N.
\]
Associated to it, each $\Phi_l$ should be generated by a $h_k(x_k,x_{k+1})$ satisfying 
\be\label{eq:cond-1}
dh_k=y_{k+1}^2/2dx_{k+1}-y_k^2/2dx_k,\quad\forall k\in\N
\ee
and commutes with $\cT$, i.e.
\be\label{eq:cond-2}
\Phi_k\circ\cT=\cT\circ\Phi_k,\quad\forall k\in\N.
\ee
If so, the mirror symmetry can be passed to $\varphi_{k+1}$.
\begin{lem}\label{lem:pass-sym}
Suppose we have (\ref{eq:cond-2}) and $\varphi_k$ is mirror symmetric, then $\varphi_{k+1}=\Phi_k\circ\varphi_k\circ\Phi_k^{-1}$ is also mirror symmetric, i.e. $\varphi_{k+1}\circ\cT\circ\varphi_{k+1}=\cT$.
\end{lem}
\begin{proof} The following analysis
\be
	\varphi_{k+1}\circ\cT &=&\Phi_{k}\circ \varphi_{k} \circ \Phi_{k}^{-1} \circ\cT\nonumber\\
	&=&\Phi_{k}\circ \varphi_{k} \circ \cT\circ\Phi_{k}^{-1}\nonumber\\
	&=&\Phi_{k}\circ \cT\circ\varphi_{k}^{-1} \circ\Phi_{k}^{-1}\nonumber\\
	&=&\cT\circ\Phi_{k}\circ\varphi_{k}^{-1} \circ\Phi_{k}^{-1}=\cT\circ\varphi_{k+1}^{-1}.\nonumber
	\ee
	directly leads to this conclusion.
\end{proof}
Notice that the transformations for $k=0,1,2$ has been made in (\ref{sr}). Inspired by that, we expect the $\varphi_k$ has a form 
\be\label{eq:general-form}
\varphi_k:\left\{
\begin{aligned}
	&\Ol x = x + \zeta(y) +  y^{k} f(x,y),\\ 
	& \Ol y = y + y^{k+1} g(x,y), \quad k\geq 3.
	\end{aligned}
	\right.
\ee
If so, conditions (\ref{eq:cond-1}) and (\ref{eq:cond-2}) almost uniquely establish the expression of $\varphi_k$.

\begin{lem}[{\sf Symplectic Constraint}]\label{lem:symp-cancel}
Suppose $\varphi_k:(x,y) \mapsto (\Ol x, \Ol y)$ is a diffeomorphism of the form (\ref{eq:general-form}) preserving the symplectic form $ydy\wedge dx$. For $y\ll1$, if the expansion of $\varphi_k$ can be expressed by 
\[
	\Ol x = x + \zeta(y) + \sum_{n=k}^\infty f_{n}(x) y^{n} , 
	\quad \Ol y = y + \sum_{n = k+1}^\infty g_{n}(x) y^{n}, 
\]
then we have
\[
(k+2) g_{k+1}(x) + f_k'(x) = 0.
\] 
\end{lem}
\begin{proof}
Notice that 
\[
\begin{aligned}
 & 	d\Ol y^2 \wedge d\Ol x \\
 	& = 
	d\big[y + g_{k+1} y^{k+1} + g_{k+2} y^{k+2} + O(y^{k+3}) \big]^2
	\\
	& \quad \quad \wedge d\big[  x + \zeta(y) + f_k y^{k} + f_{k+1} y^{k+1}  + O(y^{k+2})\big]  \\
  & = d\big[  y^2 + 2g_{k+1} y^{k+2} + 2 g_{k+2} y^{k+3} + O(y^{k+4})   \big] \\
   & \qquad \wedge \left( dx + f_k' y^{k} dx + f_{k+1}' y^{k+1}dx + (\zeta' + O(y^{k-1})dy) \right)  \\
  & = d(y^2) \wedge dx + [2g_{k+1}(k+2) + 2 f_k' ] y^{k+1}dy \wedge dx + O(y^{k+2}),
\end{aligned}
\]
therefore the $O(y^{k+1})-$term should be $0$. 
\end{proof}


\begin{lem}[{\sf Symmetric Constraint}]\label{lem:time-rev}
Suppose $\varphi_{k}$ of the form (\ref{eq:general-form}) is mirror symmetric, then the expansion of $\varphi_k$
\[
	x' = x + \zeta_k(y) + \sum_{n=k}^\infty f_n(x) y^{n} , 
	\quad y' = y + \sum_{n = k+1}^\infty g_n(x) y^{n}, 
\]
with $\zeta(y) = \sum_{n=1}^{k-1} \zeta_n y^n$ should satisfy the following:
\begin{enumerate}
\item $\zeta_k(y)$ is an odd function, namely, coefficient $\zeta_n =0$ for all even $n$;
\item If $k$ is even, then $g_{k+1}(x) =0,\; f_k(x) = 0$. 
\end{enumerate}
\end{lem}
\begin{proof}
Applying the mirror symmetry to the second component of $\varphi_k$, there holds
\[
g_{k+1}(x) (-y)^{k+4}=g_{k+1}(x) y^{k+4}.
\]
That implies for each integer $g_{k+1} =0$ if $k$ is odd. Similarly, apply the mirror symmetry to the first component of $\varphi_k$, we get the leading term equation by
\[
\zeta(y)+\zeta(-y)+[y^k+(-y)^k]f_k(x)=0,
\]
which indicates $\zeta(y)$ be odd and $f_k(x)\equiv0$ for even $k$.
\end{proof}
\begin{rmk}
This Lemma implies that each step transformation $\Phi_k$ will raise the power of $y$ by $2$ for both the two components of $\varphi_k$ in (\ref{eq:general-form}). That property decides the special form of the $\beta-$function (see Sec. \ref{s3}).
\end{rmk}
Now the only unsolved ingredient for Theorem \ref{thm:m-t} is the existence of $\Phi_k$ satisfying (\ref{eq:cond-1}) and (\ref{eq:cond-2}), which can be guaranteed by the following construction:
%
\begin{lem}[{\sf Iteration}]\label{lem:iteration}
For a billiard type diffeomorphism $\varphi_k:(x,y)\rightarrow(\Ol x, \Ol y)$ preserving $ydy\wedge dx$ and of the form 
\be\label{model}
\quad\varphi_k:
 \left\{
	\begin{aligned}
	\Ol x &= x + \zeta_k(y) + y^k f(x, y) \\
	& =  x + \zeta_k(y) + y^k f_0(x) + y^{k+1} f_1(x) + O(y^{k+2})   \\
	 \Ol y& = y + y^{k+1} g(x,y) \\
	& = y + y^{k+1} g_0(x) + y^{k+2} g_1(x) + O(y^{k+3}),\quad k\geq3,
	\end{aligned}
	\right.
\ee
we can find a coordinate change $\Phi_k^{-1}:(u,v)\rightarrow(x,y)$ with $vdv\wedge du=ydy\wedge dx$
\be\label{birk-tran}
	y = v\sqrt{1 + \frac2{k+2} v^{k-1} ([f_0]-f_0(x))}, \\
 u = x + (k+1) v^{k-1} \int_0^x \frac{[f_0]-f_0(s)}{k+2}ds ,
\ee
such that the updated normal form $\varphi_{k+1}:(u,v)\rightarrow(\Ol u,\Ol v)$ becomes
\be\label{billiard1}
\varphi_{k+1}:\left\{
	\begin{aligned}
	& \Ol u = u +\zeta_{k+1}( v)+v^{k+1}f_+(u,v)\\
	&\Ol v = v + v^{k+2}g_+(u,v).
	\end{aligned}
	\right.
\ee
Moreover, if $H(x,\Ol x)$ and $G(u,x)$  are the generating functions of $\varphi_k$ and $\Phi_k$ respectively, 
then $\varphi_{k+1}$ is generated by $H_+(u,\Ol u)$ of the form
\be\label{eq:gene-iter}
H_+(u,\Ol u)=G(u,x)-G(\Ol u, \Ol x)+H(x,\Ol x).
\ee
\end{lem}
\begin{proof}
Formally $\Phi_k^{-1}$ can be expressed by
\be\label{eq:candi-tran}
x=u+v^{k-1}a(u,v),\quad y=v+v^kb(u)+v^{k+1}e(u,v)
\ee
with $a(u,v),\;b(u)$ and $e(u,v)$ determined later on, and we can always transform (\ref{model}) into 
\be\label{formal1}
\Ol u=u+\zeta_k(v)+v^{k-1}c(u,v),\quad \Ol v=v +v^kd(u,v)
\ee
as long as $0\leq y\ll1$. The relation between $a,\ b,\ e$ and $c,\ d$ can be established by 
\be\label{high-2}
\Ol u&=&u+v^{k-1}a(u,v)+\zeta_k(v+v^kb+v^{k+1}e)\nonumber\\
& &+(v+v^kb+v^{k+1}e)^kf(u+v^{k-1}a, v+v^kb+v^{k+1}e)\nonumber\\
& &-(v+v^kd)^{k-1}a(u+\zeta_k(v)+v^{k-1}c, v +v^kd),
\ee
which can be reduced into
\be
 \Ol u&=&u +\zeta_k(v)+v^k\Big{[}\zeta_k'(v)b(u)+f_0(u)-a_u(u,0)\frac{\zeta_k}{v}\Big{]}\nonumber\\
 & &+v^{k+1}\Big{[}f_1-a_{uu}\frac{\zeta_k^2}{2v^2}-a_{uv}(u,0)\frac{\zeta_k}{v}\Big{]}\nonumber\\
& &-v^{2k-2}\Big{[}a_uc+d(k-1)a\Big{]}+h.o.t.
\ee
If we choose $a(u, v), b(u)$ so that
\todoin{
\be
	-a_u(u, 0) + f_0(u) - \int_0^1 f_0(u)du + b(u) = 0
\ee
}
\noindent then the first component of $\varphi_{k+1}$ will become 
\[
	\Ol u = u +\zeta_{k+1}( v)+  f_{+}(u,v) v^{k+1}
\]
with
\todoin{
\[
\zeta_{k+1}(v)=\zeta_k(v)+[f_0]v^k,  \quad [f_0] = \int_0^1 f_0(u) du.
\]
}
In previous estimate we used conclusions $\zeta'_k(0)=1,\;\zeta''_k(0)=0$, $k\geq3$. On the other hand, 
\ben
& &	\barv + \barv^k b(\baru) +\bar v^{k+1}e(\bar u,\bar v)\\
&=& v + v^k b(u) + v^{k+1}e(u,v)+v^{k+1} g_0(v) + O(v^{k+2}),
\een
because 
\[
b(\baru) - b(u) = b'(u) v + O(v^2),
\]
so we formally get 
\[
	\barv = v + v^{k+1}( - b'(u)+  g_0(u)) + O(v^{k+2}) := v + v^{k+2} g_+(u, v) 
\]
once
\todoin{
\be\label{raise1}
- b'(u)+  g_0(u)=0.
\ee
}
%
\noindent Recall that $a(u,v), b(u)$ have to be chosen to preserve $v dv \wedge du = y dy \wedge dx$. This is available by taking $S(x,v)$ satisfying 
\be
\begin{aligned}
	dS(x, v) &= \frac12 y^2 dx - \frac12 v^2 du + \frac12  d\left( (u-x) v^2 \right)  \\
  &	= \frac12 y^2 dx - \frac12 v^2 dx  + (u-x)v dv  .
\end{aligned}
\ee
That imposes
\be\label{eq:gene-tran}
	y^2 = v^2 + 2 S_x(x, v), \quad u = \frac1{v} S_v(x,v) + x,
\ee
therefore, we have to choose
\[
S_x(x, v) =  v^{k+1} b(x)
\]
 due to (\ref{eq:candi-tran}) and accordingly 
$e(u,v)=O(v^{k-2})$. On the other side, first equation of (\ref{eq:candi-tran}) implies
\[
S_v(u+v^{k-1}a(u,v),v)=-v^ka (u,v).
\]
This equation is always solvable by the {\sf Implicit Function Theorem} for suitably small $v\ll1$. Here we just need to choose 
	\todoin{
	\[
	S(x, v): = v^{k+1} B(x)=-\frac{v^{k+1}}{k+2}\int_0^xf_0(s)-[f_0]ds.
\]
}
\noindent That conversely implies
\begin{equation}\label{eq:b}
\left\{
\begin{aligned}
 b(u)&= \frac{[f_0]-f_0(u)}{k+2},& \\
	a(u, 0) &=  \frac{k+1}{k+2} \int_0^u f_0(s)-[f_0]ds&
	\end{aligned}
	\right.
\end{equation}
The coordinate change $\Phi_k$ is now established by 
\be
\Phi_k:\left\{
\begin{aligned}
	&y = v\sqrt{ 1+ \frac2{k+2} v^{k-1} ([f_0]-f_0(x))},\\
	& u = x + (k+1) v^{k-1} B(x) 
	\end{aligned}
	\right.
\ee
and the new map $\varphi_{k+1}$ becomes
\be
\varphi_{k+1}:\left\{
	\begin{aligned}
	&	\baru = u +\zeta_{k+1}( v)+  v^{k+1}f_+(u,v) \\
	& \barv = v + v^{k+2} g_+(u, v).
	\end{aligned}
	\right.
\ee
Besides, 
\[
d(u,v)=O(v^2),\quad e(u,v)=O(v^{k-2})
\]
are actually higher order irrelevant term in aforementioned computation. 
So we finally establish the expression of $\varphi_{k+1}$ as (\ref{billiard1}).
Moreover, we can find a generating function corresponding to the coordinate transformation:
\[
dG(u,x)=y^2/2dx-v^2/2du
\]
satisfying
\be
G(u,x)&=&(x-u)w+S(x,v)\nonumber\\
&=&\frac{1-k}{2}(k+1)^{-\frac{k+1}{k-1}}B(x)^{-\frac{2}{k-1}}(u-x)^{\frac{k+1}{k-1}}.
\ee
Due to Remark \ref{rmk:sym-inv} we get the 
 generating function $H_+(u,\bar u)$ as (\ref{eq:gene-iter}).
\end{proof}
\vspace{10pt}

\begin{proof} {\it of Theorem \ref{thm:m-t}:}
Starting from $k=3$, now we apply Lemma \ref{lem:iteration} to $\varphi_3$, then get $\varphi_4=\Phi_3\circ\varphi_3\circ\Phi_3^{-1}$ of the form (\ref{billiard1}). Notice that
\[
\Phi_3\circ \cT=\cT\circ\Phi_3
\]
due to (\ref{birk-tran}). Benefit from Lemma \ref{lem:time-rev}, $\varphi_{4}$ can be further improved into
\be
\varphi_5:\left\{
	\begin{aligned}
	& \baru = u + \zeta_4(v)+ O(v^5) \\
	& \barv = v + O(v^6),
	\end{aligned}
	\right.
\ee
so we raise $k$ by 2 actually in one step iteration!!! Notice that the mirror symmetry is preserved for $\varphi_5$ because (\ref{birk-tran}) commutes with $\cT$ for odd $k$, then due to Lemma \ref{lem:pass-sym} $\varphi_5$ is mirror symmetric.
Therefore, the iteration can be continued for infinitely many steps. If we denote by 
\[
\Phi_\infty:=\prod_{\substack{k=3\\k\text{ is odd}}}^{+\infty}\Phi_k
\]
the composition of infinitely many transformations, then $\varphi_\infty=\Phi_\infty\circ\varphi_3\circ\Phi_\infty$ has a normal form as (\ref{eq:birk-norm-form}).

Apparently when $\partial\Om$ is $C^\infty$ (or $C^\om$) smooth, all these $\varphi_k$ and $\Phi_k$ are $C^\infty$ smooth in $\A$ as well. Compositing these for infinitely many times then we get $\varphi_\infty$ is $C^\infty$ smooth with a normal form given by (\ref{eq:tar-thm})  as $y\rightarrow0$. 
\end{proof}

\section{Spectral Invariants and $\beta-$function}\label{s3}


In this section we solve  the $\beta-$function relying on the normal form given in Theorem \ref{thm:m-t}, and explore some related corollaries.

\begin{proof}{\it of Corollary \ref{cor:1}:}
Recall that (\ref{eq:tar-thm}) preserves the symplectic form $ydy\wedge dx$, if we take $l:=y^2/2$, then (\ref{eq:tar-thm}) can be transformed into 
\[
\chi:\A\rightarrow\A,\quad\text{via }\left\{
\begin{aligned}
&x^+=x+\zeta_\infty(\sqrt{2l}),\\
&l^+=l
\end{aligned}
\right.
\]
which is an exactly symplectic twist map. As we know $\chi$ is integrable, so we can slove the `formal Hamiltonian' 
\[
H(l)=\int\zeta_\infty(\sqrt{2l})dl=\frac{2\sqrt 2}{3}l^{3/2}+\sum_{i=1}^{+\infty}\frac{2c_{2i+!}}{2i+3}\sqrt{2^{2i+1}}l^{\frac{2i+3}{2}}
\]
of which $\chi$ equals the time-1 map. Moreover, since $H(l)$ is integrable, we can easily solve the `formal $\alpha-$function' of it, $\alpha(c)=H(c)$, which works as the convex conjugation of $\beta(h)$ \cite{Ma2}. Due to the Legendre transformation, we have
\be
\beta(h)&=&\max_{c\in H^1(\T,\R)}\langle c,h\rangle-\alpha(c)\nonumber\\
&=&\frac{\sqrt2}3c^{3/2}+\sum_{i=1}^{+\infty}\frac{2i+1}{2i+3}c_{2i+1}\sqrt{2^{2i+1}}c^{\frac{2i+3}{2}}\nonumber
\ee
with 
\[
\sqrt{2c}=\zeta_\infty^{-1}(h)=h+\sum_{i=1}^{+\infty} d_{2i+1}h^{2i+1}
\]
formally as $h\ll1$.
\end{proof}

\begin{proof}{\it of Corollary \ref{cor:2}:}
Due to (\ref{generating}), we can solve the first two coefficients in (\ref{beta-formal}) by 
\[
\beta_1=1=\text{length}(\partial\Om),
\]
\[
\beta_3=\frac14\Big(\int_0^1\rho^{-2/3}(s)ds\Big)^3.
\]
By a similar estimate as the {\sf isoperimetric inequality} we get $\beta_3+\pi^2\beta_1\leq 0$, where the equality holds only for $\rho(s)\equiv 2\pi$; That indicates $\partial\Om$ is a circle. 
\end{proof}

\end{document}